\documentclass[11pt,twoside]{amsart}
\usepackage{amsmath, amsthm, amscd, amsfonts, amssymb, graphicx, color}
\usepackage[bookmarksnumbered, plainpages]{hyperref}
\newtheorem{thm}{Theorem}[section]

\newtheorem{rem}[thm]{\bf{Remark}}

\numberwithin{equation}{section}
\def\pn{\par\noindent}

\begin{document}


\title{A note on noninner automorphisms of order $p$ for  finite $p$-groups of coclass $2$}
\author{A. Abdollahi$^*$, S. M. Ghoraishi and  B. Wilkens}

\thanks{{\scriptsize
\hskip -0.4 true cm MSC(2010): Primary: 20D45 Secondary: 20E36.
$*$Corresponding author}}
\maketitle


\begin{abstract}  In this note, the existence of noninner automorphisms of order $2$ for finite $2$-groups of coclass $2$ is proved.
Combining our result with a recent one due to Y. Guerboussa and  M. Reguiat (see \href{http://arxiv.org/pdf/1301.0085v1.pdf}{arXiv:1301.0085}), we prove that every finite  $p$-group of coclass $2$ has a noninner automorphism of order $p$ leaving  the center elementwise fixed.
\end{abstract}



%


\section{\bf Introduction}
\vskip 0.4 true cm

Let $G$ be a finite nonabelian $p$-group. A longstanding conjecture asserts that every finite nonabelian $p$-group admits a noninner
automorphism of order $p$. This conjecture has been settled  for various  classes of $p$-groups $G$ as follows:
\begin{itemize}
\item if $G$ is regular \cite{S, DS};
\item if $G$ is nilpotent of class $2$ or $3$ \cite{A,AGW,L};
\item if the commutator subgroup of $G$ is cyclic \cite{jam};
\item if $G/Z(G)$ is powerful, \cite{AB};
\item if $C_G(Z(\Phi(G)))\neq \Phi(G)$ \cite{DS};
\item if $p$ is odd and $G$ is of coclass $2$ \cite{yass};
\end{itemize}
It is proved that $G$ has often a noninner automorphism of order $p$ leaving the center $Z(G)$  or Frattini subgroup $\Phi(G)$ of $G$ elementwise fixed.

It is proved in \cite{AG} that if  $G$ is regular or  nilpotent of class $2$ or if  $G/Z(G)$ is powerful or  $C_G(Z(\Phi(G)))\neq \Phi(G)$, then $G$ has a noninner automorphism of order $p$ acting trivially on $Z(G)$.

As we mentioned above, the conjecture is valid for $p$-groups of coclass $2$ if $p>2$.
In this note, we confirm the validity of the conjecture for groups of coclass $2$ by proving the following:
\begin{thm}\label{main}
All finite $2$-groups of coclass $2$ have  noninner automorphisms of order $2$ leaving  the center  elementwise fixed.
\end{thm}
As a consequence of \cite[Theorem 5.1]{yass} and Theorem \ref{main}, we prove that:
\begin{thm}\label{main1}
Every finite  $p$-group of coclass $2$ has a noninner automorphism of order $p$ leaving  the center elementwise fixed.
\end{thm}

\section{\bf Proofs}
\vskip 0.4 true cm

 Let $G$ be a nonabelian finite $p$-group. We use the following facts in the proof.

\begin{rem}[{\cite[Theorem]{DS}}]\label{DS}
    If $C_G(Z(\Phi(G)))\neq \Phi(G)$, then $G$  has a noninner automorphism of order $p$ leaving   the Frattini subgroup $\Phi(G)$ elementwise fixed.
\end{rem}
\begin{rem}[{\cite[Lemma 2.2]{AB}}]\label{AB-lem}
If $d(Z_2(G)/Z(G))\neq d(G)d(Z(G))$ then $G$ has a  noninner automorphism of order $2$ leaving the Frattini  subgroup $\Phi(G)$ of $G$ elementwise fixed.
\end{rem}
\begin{rem}\label{jam}
 If $G$ has cyclic commutator subgroup, then  $G$ has a noninner automorphism of order $p$ leaving $\Phi(G)$ elementwise fixed whenever
$p > 2$, and leaving either $\Phi(G)$ or $Z(G)$ elementwise fixed whenever $p = 2$.
\end{rem}
\begin{rem}[{\cite[Corollary 3.3.4 (iii)]{leed}}]\label{leed}
A $2$-group of maximal class is either the dihedral, semidihedral or the quaternion group of order at least $2^4$.
\end{rem}
Recall that a group $G$ is called capable if it is the group of inner automorphisms
of some group, that is, if there exists a group $H$ with $H/Z(H)\cong G$.
\begin{rem}[{\cite[Theorem]{Sh}}]\label{capable}
 $Q_{2^n}$, the generalized quaternion group of order $2^n$  ($n > 2$), and $SD_{2^n}$, the semidihedral
group of order $2^n$  ($n > 3$), cannot be normal subgroups of a capable group.
\end{rem}

\noindent {\bf Proof of Theorem \ref{main}.}\\
  Let $G$ be finite $2$-group of coclass $2$ and suppose on the contrary that $G$ has no noninner automorphism of order $2$.
By Remark \ref{AB-lem} we may assume that $d(Z_2(G)/Z(G))=d(G)d(Z(G))$.
 Since $G$ is of  coclass $2$, we get $|Z(G)|=2$,
$d(G)=2$ and $Z_2(G)/Z(G)$ is elementary abelian of rank $2$. Thus
$x^2\in Z(G)$, for every $x\in Z_2(G)$ and  $[Z_2(G),\Phi(G)]=1$.
On the other hand it follows from Remark \ref{DS} that $C_G(Z(\Phi(G)))= \Phi(G)$.
Therefore   $Z_2(G)\leq Z(\Phi(G))$. Hence  $Z_2(G)$ is a noncyclic  abelian subgroup of order $2^3$.

Let $\widetilde G=G/Z_2(G)$ and let the nilpotency class $cl(\widetilde G)$ of  $\widetilde G$ denote by $k$. Then $cl(G)=k+2$.
and  $\gamma_i(G)Z_2(G)=Z_{k+3-i}(G)$ for all $i\in\{1,\dots,k\}$. In addition,  we may suppose that $k\geq 2$.

  Note that $\tilde G$ is of maximal class. Therefore by Remark  \ref{leed}, $\widetilde G$ is either the dihedral, semidihedral or  quaternion group of order at least $2^4$.

Since  $\widetilde G$ is capable,  Remark \ref{capable} implies that
$\widetilde G\cong D_{2^{k+1}}$ the dihedral  group of order $2^{k+1}$. Then $G=\langle x,y\rangle$ such that $\widetilde G=\langle \tilde x,\tilde y\mid \tilde y^{2^k}=\tilde 1,\, \tilde x^2=\tilde 1,\, \tilde y^{\tilde x}={\tilde y}^{-1}\rangle$. We obtain
\begin{itemize}
  \item
  $[x,y]=y^{2}t$ for some $t\in Z_2(G)$,
  \item
   $[x,y^2]\in y^4Z(G)$,
  \item
  $[x,y^4]=y^8$ and so $\langle y^4\rangle\lhd G$.
  \item
  $\gamma_i(G)\leq \langle y^{2^{i-1}}\rangle Z(G)\leq \langle y^{2^{i-1}}\rangle$, for $3\leq i\leq k+3$.
 \end{itemize}
Now since $\gamma_{k+2}(G)\leq Z(G)$, it follows that $Z(G)=\langle y^{2^{k+1}}\rangle$.
Thus
$Z_2(G)=\langle y^{2^{k}}\rangle\times\langle v\rangle$, for some element $v$ of order $2$.

 Thus $t=y^{2^ki}v^j$ for some integers $i,j$. Then  $y^2t=y^{2(1+2^{k-1}i)}v^j$ and  $(y^2t)^2=y^{4(1+2^{k-1}i)}$. Therefore $G'=\langle [x,y], \gamma_3(G)\rangle\leq \langle y^2t, y^4\rangle=\langle y^2t\rangle$ is cyclic. Now it follows from Remark \ref{jam} that   $G$ has a noninner automorphism of order $2$  leaving either $\Phi(G)$ or $Z(G)$ elementwise fixed, a contradiction. This completes the proof. $\hfill \Box$ \\

Now we can prove that
\begin{thm}\label{main0}
Every finite  $p$-group of coclass $2$ has a noninner automorphism of order $p$.
\end{thm}
\begin{proof}
  It follows from Theorem \ref{main} and \cite[Theorem 5.1]{yass} that  every finite $p$-group of coclass $2$ has a noninner automorphism of order $p$.
\end{proof}
\noindent{\bf Proof of Theorem \ref{main1}.}\\
Let $G$ be a finite $p$-group of coclass $2$. Thus $|Z(G)|\in\{p,p^2\}$. \\
If $|Z(G)|=p$ then every $p$-automorphism of $G$ leaving $Z(G)$ elementwise fixed. Thus Theorem \ref{main0} completes the proof.\\
Thus we may assume that $|Z(G)|=p^2$.

If $Z(G)\not\leq \Phi(G)$, then for some $g\in Z(G)$ and  some maximal subgroup $M$ of $G$ we have $g\notin M$. Thus  $cl(M)=cl(G)$ and so  $M$ is of  maximal class. It follows from \cite[Corollary 2.4.]{AB} that  $M$ has a noninner automorphism $\alpha$ of order $p$ leaving $\Phi(M)$ elementwise fixed. Now the map given by $g\mapsto g$ and $m\mapsto m^\alpha$ for all $m\in M$,
 determines a noninner automorphism of order $p$ leaving  $Z(G)$ elementwise fixed.
 Therefore we may assume that $Z(G)\leq \Phi(G)$. It follows from Remark \ref{AB-lem} that $G$ has a noninner automorphism $\alpha$ of order $p$ leaving $\Phi(G)$ elementwise fixed. Since $Z(G)\leq \Phi(G)$, $\alpha$ leaves $Z(G)$ elementwise fixed. This completes the proof. $\hfill \Box$\\



\bigskip
\bigskip

{\footnotesize \pn{\bf Alireza Abdollahi}\; \\ Department of Mathematics, University of Isfahan, Isfahan 81746-73441, Iran\\
{\tt Email: a.abdollahi@math.ui.ac.ir}\\

{\footnotesize \pn{\bf S. M. Ghoraishi}\; \\ The Isfahan Branch of the School of Mathematics, Institute for Research in 
Fundamental Sciences (IPM) Isfahan, Iran \\
{\tt Email: ghoraishi@ipm.ir}\\

{\footnotesize \pn{\bf B. Wilkens}\; \\ Department of Mathematics, University of Botswana, Private Bag 00704, Gaborone, Botswana \\
{\tt Email: wilkensb@mopipi.ub.bw}\\


\begin{thebibliography}{99}
\bibitem{AB} A. Abdollahi, Powerful $p$-groups have noninner automorphisms of order $p$ and some cohomology,
{\em J.  Algebra}, \textbf{323} (2010) 779--789.

\bibitem{A} A. Abdollahi, Finite $p$-groups of class $2$ have noninner automorphisms of order $p$, {\em J.  Algebra}, \textbf{312} (2007) 876--879.

\bibitem{AG}
A. Abdollahi and M. Ghoraishi, Noninner automorphisms of finite $p$-groups leaving the center elementwise fixed,
{\em Int. J. Group Theory}, \textbf{2} no. 4 (2013) 17-20.

\bibitem{AGW} A. Abdollahi, M. Ghoraishi and B. Wilkens, Finite $p$-groups of class 3 have noninner automorphisms of order $p$,
{\em Beitr Algebra Geom.},  {\bf 54} no. 1 (2013) 363-381.



\bibitem{DS} M. Deaconescu and G. Silberberg, Noninner automorphisms of order $p$ of finite $p$-groups, {\em J. Algebra} \textbf{250} (2002) 283--287.


\bibitem{G} W. Gasch\"{u}tz, Nichtabelsche  $p$-Gruppen besitzen \"{a}ussere $p$-Automorphismen, {\em J. Algebra}, \textbf{4} (1966) 1--2.


\bibitem{SMG} S. M. Ghoraishi,  A note on automorphisms of finite $p$-groups, {\em Bull. Aust. Math. Soc.}, \textbf{87} (2013) 24--26.

\bibitem{leed}
C. R. Leedham-Green and S. McKay, The structure of Groups of prime power order, London Math. Soc.

\bibitem{yass}
Y. Guerboussa and  M. Reguiat, Some automorphism groups of finite $p$-groups, \href{http://arxiv.org/pdf/1301.0085v1.pdf}{arXiv:1301.0085}.

\bibitem{jam}
A. R. Jamalli and M. Viseh, On the existence of noinner automorphisms of order two in finite $2$-groups,  {\em Bull. Aust. Math.  Soc.}, \textbf{87} (2013) 278--287.

\bibitem{L} H. Liebeck, Outer automorphisms in nilpotent  $p$-groups of class 2, {\em J. London Math. Soc.} \textbf{40} (1965)  268--275.



\bibitem{S} P. Schmid,  A cohomological property of regular $p$-groups, {\em Math. Z.}, \textbf{175} (1980) 1--3.

\bibitem{shab}
 M. Shabani-Attar, Existence of noninner automorphisms of order $p$ in some finite $p$-groups,  {\em Bull. Aust.  Math.  Soc.}, \textbf{87} (2013) 272--277.

\bibitem{Sh} S. Shahriari,  On normal subgroups of capable groups, {\em Arch. Math. (Basel)},  \textbf{48} (1987) 193-198.

\end{thebibliography}
\end{document}